\documentclass[12pt]{amsart}
\usepackage{fullpage,url}

\DeclareFontEncoding{OT2}{}{} 

\usepackage{color}

\newcommand{\defi}[1]{\textsf{#1}} 

\newenvironment{romanenum}{\hfill \begin{enumerate} }{\end{enumerate}}

\newcommand{\C}{{\mathbb C}}

\newcommand{\G}{{\mathbb G}}

\newcommand{\Q}{{\mathbb Q}}
\newcommand{\R}{{\mathbb R}}
\newcommand{\Z}{{\mathbb Z}}

\newcommand{\calG}{{\mathcal G}}
\newcommand{\calH}{{\mathcal H}}

\newcommand{\calT}{{\mathcal T}}

\newcommand{\OO}{{\mathcal O}}

\DeclareMathOperator{\reg}{reg}

\DeclareMathOperator{\rk}{rk}

\DeclareMathOperator{\Lie}{Lie}
\DeclareMathOperator{\Hom}{Hom}

\DeclareMathOperator{\Aut}{Aut}
\DeclareMathOperator{\Gal}{Gal}

\DeclareMathOperator{\Res}{Res}

\newcommand{\GL}{\operatorname{GL}}

\newcommand{\isom}{\simeq}

\newcommand{\del}{\partial}
\newcommand{\intersect}{\cap} 
\newcommand{\tensor}{\otimes}

\newtheorem{theorem}{Theorem}[section]
\newtheorem{lemma}[theorem]{Lemma}
\newtheorem{corollary}[theorem]{Corollary}
\newtheorem{proposition}[theorem]{Proposition}

\theoremstyle{definition}
\newtheorem{definition}[theorem]{Definition}
\newtheorem{question}[theorem]{Question}

\newtheorem{examples}[theorem]{Examples}

\theoremstyle{remark}
\newtheorem{remark}[theorem]{Remark}

\usepackage[
        backref,
        pdfauthor={Bjorn Poonen}, 
]{hyperref}

\usepackage[alphabetic,backrefs,lite]{amsrefs} 

\begin{document}

\title{The $p$-adic closure of a subgroup of rational points on a commutative algebraic group}
\subjclass[2000]{Primary 14L10; Secondary 11E95, 14G05, 14G20, 22E35}
\keywords{$p$-adic logarithm, rational points, algebraic group, Chabauty's method, Leopoldt's conjecture}
\author{Bjorn Poonen}
\thanks{This research was supported by NSF grant DMS-0301280.}
\address{Department of Mathematics, University of California, 
        Berkeley, CA 94720-3840, USA}
\email{poonen@math.berkeley.edu}
\urladdr{http://math.berkeley.edu/\~{}poonen}
\date{November 30, 2007}

\begin{abstract}
Let $G$ be a commutative algebraic group over $\Q$.
Let $\Gamma$ be a subgroup of $G(\Q)$ 
contained in the union of the compact subgroups of $G(\Q_p)$.
We formulate a guess for the dimension of the closure of $\Gamma$
in $G(\Q_p)$, and show that its correctness for certain tori
is equivalent to Leopoldt's conjecture.
\end{abstract}

\maketitle

\section{Introduction}\label{S:introduction}

\subsection{Notation}

Let $\Q$ be the field of rational numbers.
Let $p$ be a prime,
and let $\Q_p$ be the corresponding completion of $\Q$.
Let $\Z_p$ be the completion of $\Z$ at $p$.
If $K$ is a number field, then $\OO_K$ is the ring of the integers,
and for any finite set $S$ of places, $\OO_{K,S}$ is the ring of $S$-integers.
If $G$ is a group (or group scheme),
then $H \le G$ means that $H$ is a subgroup (or subgroup scheme) of $G$.
If $A$ is an integral domain with fraction field $K$,
and $M$ is an $A$-module,
then $\rk_A M$ is the dimension of the $K$-vector space $M \tensor_A K$;
we write $\rk$ for $\rk_\Z$.

\subsection{The logarithm map for a $p$-adic Lie group}

Let $G$ be a finite-dimensional commutative
Lie group over $\Q_p$ (see \cite{BourbakiLie1-3}*{III.\S1} for terminology).
The Lie algebra $\Lie G$ is the tangent space of $G$ at the identity.
So $\Lie G$ is a $\Q_p$-vector space of dimension $\dim G$.
Let $G_f$ be the union of the compact subgroups of $G$.
By \cite{BourbakiLie1-3}*{III.\S7.6}, 
$G_f$ is an open subgroup of $G$,
and there is a canonical homomorphism
\[
        \log\colon G_f \to \Lie G,
\]
defined first on a sufficiently small compact open subgroup 
by formally integrating translation-invariant $1$-forms,
and then extended by linearity.
Moreover, $\log$ is a local diffeomorphism, and its kernel 
is the torsion subgroup of $G_f$.
It behaves functorially in $G$.

\begin{examples}
\begin{romanenum}
\item 
If $G=\Q_p$ (the additive group), then $G_f=\Q_p$, 
and $\log$ is an isomorphism.
In this example, $G_f$ is not compact.
\item 
If $G=\Q_p^\times$, then $G_f=\Z_p^\times$.
\item
If $G=A(\Q_p)$ for an abelian variety $A$ over $\Q_p$, then $G_f=G$.
\end{romanenum}
\end{examples}

\subsection{Dimension of an analytic subgroup}

Let $\Gamma$ be a finitely generated subgroup of $G_f$.
Then $\log \Gamma \subseteq \Lie G$ 
is a finitely generated abelian group of the same rank.
The closure $\overline{\log \Gamma} = \log \overline{\Gamma}$
with respect to the $p$-adic topology equals the 
$\Z_p$-submodule of $\Lie G$ spanned by $\log \Gamma$,
so it is a finitely generated $\Z_p$-module.
Define
\[
        \dim \overline{\Gamma} := \rk_{\Z_p} \overline{\log \Gamma}.
\]
This agrees with the dimension of $\overline{\Gamma}$ viewed
as a Lie group over $\Q_p$.

\subsection{Rational points}

Now let $G$ be a commutative group scheme of finite type over $\Q$.
Fix a prime $p$.
Define $G(\Q)_f:=G(\Q) \intersect G(\Q_p)_f$.
We specialize the previous sections to the Lie group $G(\Q_p)$
and to a finitely generated subgroup $\Gamma$ of $G(\Q)_f$.
Our goal is to predict the value of $\dim \overline{\Gamma}$.

\subsection{Applications}

The value of $\dim \overline{\Gamma}$ is important for a few reasons:
\begin{romanenum}
\item If $C$ is a curve of genus $g \ge 2$ over $\Q$
embedded in its Jacobian $J$, 
then the condition $\dim \overline{J(\Q)} < g$
is necessary for the application of Chabauty's method,
which attempts to calculate $C(k)$ 
or at least bound its size \cites{Chabauty1941,Coleman1985chabauty}.
\item Leopoldt's conjecture on $p$-adic independence of units 
in a number field predicts $\dim \overline{\Gamma}$ in a special case:
see Corollary~\ref{C:Leopoldt} in Section~\ref{S:tori}.
Leopoldt's conjecture is important 
because it governs the abelian extensions of $K$ of $p$-power degree.
\end{romanenum}

\subsection{Outline of the paper}

Section~\ref{S:dimension functions}
axiomatizes some of the properties of $\dim \overline{\Gamma}$
in order to identify possible candidates for its value.
Section~\ref{S:guess}
defines a ``maximal'' function $d(\Gamma)$ satisfying the same axioms
and Question~\ref{Q:over Q} asks whether 
it always equals $\dim \overline{\Gamma}$.
Section~\ref{S:shared} shows that $\dim \overline{\Gamma}$
and $d(\Gamma)$ share many other properties.
Section~\ref{S:tori} computes $d(\Gamma)$
for subgroups of integer points on tori,
and shows that a positive answer to
Question~\ref{Q:over Q} for certain tori would imply Leopoldt's conjectures.
We end with further open questions.

\section{Dimension functions}
\label{S:dimension functions}

Let $\calG$ be the set of pairs $(G,\Gamma)$ 
where $G$ is a commutative group scheme of finite type over $\Q$
and $\Gamma$ is a finitely generated subgroup of $G(\Q)_f$.

\begin{definition}
A \defi{dimension function} is a function $\del \colon \calG \to \Z_{\ge 0}$ 
satisfying
\begin{enumerate}
\item If $\Gamma \le H(\Q)_f$ for some subgroup scheme $H \le G$,
then $\del(H,\Gamma)=\del(G,\Gamma)$.  
(Because of this, 
we generally write $\del(\Gamma)$ instead of $\del(G,\Gamma)$.)
\item $\del(\Gamma) \le \rk \Gamma$.
\item If $H \le G$ and $\Gamma''$ is the image of $\Gamma$ in $(G/H)(\Q)_f$,
then $\del(\Gamma) \le \dim H + \del(\Gamma'')$.
\end{enumerate}
\end{definition}

\begin{proposition}
The expression $\dim \overline{\Gamma}$ is a dimension function.
\end{proposition}

\begin{proof}
Since $H(\Q_p)$ is closed in $G(\Q_p)$,
the closure of $\Gamma$ in $H(\Q_p)$ equals the closure of $\Gamma$
in $G(\Q_p)$;
therefore~(1) holds.
The fact that $\overline{\log \Gamma}$ is 
the $\Z_p$-submodule spanned by $\log \Gamma$
gives the middle step in
\[
        \dim \overline{\Gamma} 
        = \rk_{\Z_p} \overline{\log \Gamma} 
        \le \rk (\log \Gamma) \le \rk \Gamma,
\]
so~(2) holds.
Finally, by continuity, the image of $\overline{\Gamma}$ in $(G/H)(\Q_p)$
equals $\overline{\Gamma''}$, so we have an exact sequence
\[
        0 \to \overline{\Gamma} \intersect H \to \overline{\Gamma} \to \overline{\Gamma''} \to 0.
\]
Taking dimensions of $p$-adic Lie groups and observing that the group
on the left has dimension at most $\dim H$ yields~(3).
\end{proof}

\section{The guess}
\label{S:guess}

With notation as before,
define
\[
        d(G,\Gamma) := \inf_{H \le G} 
                \left( \dim H + \rk \Gamma - \rk(\Gamma \intersect H) \right),
\]
where the infimum is over all subgroup schemes $H \le G$.

\begin{proposition}
\label{P:d is dimension}
The function $d$ is a dimension function,
and any dimension function $\del$ satisfies $\del \le d$.
\end{proposition}

\begin{proof}
First we check that $d$ is a dimension function:
\begin{enumerate}
\item 
Suppose $G' \le G$ and $\Gamma \le G'(\Q)_f$.
If $H \le G'$ then the subgroup $H' := H \intersect G'$
satisfies $\dim H' \le \dim H$
and $\Gamma \intersect H' = \Gamma \intersect H$,
so
\[
        \dim H' + \rk \Gamma - \rk(\Gamma \intersect H')
        \le \dim H + \rk \Gamma - \rk(\Gamma \intersect H).
\]
Therefore the infimum in the definition of $d(G,\Gamma)$ is
attained for some $H \le G'$, so $d(G,\Gamma)=d(G',\Gamma)$.
\item 
The $H=\{0\}$ term in the infimum is $0 + \rk \Gamma  - 0$,
so $d(\Gamma) \le \rk \Gamma$.
\item
Let $K''$ be the subgroup of $G/H$ realizing the infimum 
defining $d(\Gamma'')$.
Let $K$ be the inverse image of $K''$ under $G \to G/H$.
Then $\Gamma'' \intersect K''$ is a homomorphic image of $\Gamma \intersect K$,
so $\rk(\Gamma'' \intersect K'') \le \rk(\Gamma \intersect K)$ and
\begin{align*}
        d(\Gamma) 
        &\le \dim K + \rk \Gamma - \rk(\Gamma \intersect K) \\
        &= \dim H + \dim K'' + \rk \Gamma - \rk(\Gamma \intersect K) \\
        &\le \dim H + \dim K'' + \rk \Gamma - \rk(\Gamma'' \intersect K'') \\
        &= \dim H + d(\Gamma'').
\end{align*}
\end{enumerate}

Now we check that any dimension function $\del$ satisfies $\del \le d$.
If $H \le G$ and $\Gamma''$ is the image of $\Gamma$ in $(G/H)(\Q)$,
then properties (3) and~(2) for $\del$ and the isomorphism
$\Gamma'' \isom \Gamma/(\Gamma \intersect H)$
yield
\[
        \del(\Gamma) 
        \le \dim H + \del(\Gamma'')
        \le \dim H + \rk \Gamma''
        =  \dim H + \rk \Gamma - \rk(\Gamma \intersect H).
\]
This holds for all $H$, so $\del(\Gamma) \le d(\Gamma)$.
\end{proof}

\begin{corollary}
We have $\dim \overline{\Gamma} \le d(\Gamma)$.
\end{corollary}

Proposition~\ref{P:d is dimension}
shows that the function $d(\Gamma)$ 
gives the largest guess for $\dim \overline{\Gamma}$ 
compatible with the elementary inequalities 
based on rank and the dimension of the group.
Therefore we ask:

\begin{question}
\label{Q:over Q}
Does $\dim \overline{\Gamma} = d(\Gamma)$ always hold?
\end{question}

In other words, are rational points $p$-adically independent 
whenever dependencies are not forced by having a subgroup of too high rank
inside an algebraic subgroup?

\section{Further properties shared by both dimension functions}
\label{S:shared}

\begin{proposition}
\label{P:shared}
Let $\del(\Gamma)$ denote either $\dim \overline{\Gamma}$ or $d(\Gamma)$.
Then
\begin{romanenum}
\item
If $\Gamma' \le \Gamma \le G(\Q)_f$,
then $\del(\Gamma') \le \del(\Gamma)$.
\item If $G \to G''$ is a homomorphism and $\Gamma \le G(\Q)_f$, 
then the image $\Gamma''$ in $G''(\Q)$ is contained in $G''(\Q)_f$
and $\del(\Gamma'') \le \del(\Gamma)$.
\item
If $G \to G''$ has finite kernel and $\Gamma \le G(\Q)$,
then $\Gamma$ is contained in $G(\Q)_f$ 
if and only if its image $\Gamma''$ in $G''(\Q)$
belongs to $G''(\Q)_f$;
in this case, $\del(\Gamma)=\del(\Gamma'')$.
\item
Suppose that $\Gamma_1,\Gamma_2$ are 
\defi{commensurable} subgroups of $G(\Q)$;
i.e., $\Gamma_1 \intersect \Gamma_2$ 
has finite index in both $\Gamma_1$ and $\Gamma_2$.
Then $\Gamma_1 \le G(\Q)_f$ if and only if $\Gamma_2 \le G(\Q)_f$;
in this case, $\del(\Gamma_1) = \del(\Gamma_2)$.
\item
If $\Gamma_i \le G_i(\Q)_f$ for $i=1,2$,
then $\Gamma_1 \times \Gamma_2 \le (G_1 \times G_2)(\Q)_f$
and $\del(\Gamma_1 \times \Gamma_2) =\del(\Gamma_1) + \del(\Gamma_2)$.
\item
If $\Gamma_1,\Gamma_2 \le G(\Q)_f$,
then $\del(\Gamma_1 + \Gamma_2) \le \del(\Gamma_1) + \del(\Gamma_2)$.
\item
If $\rk \Gamma = 1$, then $\del(\Gamma)=1$.
\item
If $G \isom \G_a^n$, then $\del(\Gamma) = \rk \Gamma$.
\end{romanenum}
\end{proposition}

\begin{proof}
\begin{romanenum}
\item
For $\del(\Gamma):=\dim \overline{\Gamma}$ the result is obvious.
For $d(\Gamma)$ it follows since
$\rk \Gamma  - \rk(\Gamma \intersect H)$
equals the rank of the image of $\Gamma$ in $G/H$.
\item
We have $\Gamma'' \le G''(\Q)_f$ by functoriality.
For $\dim \overline{\Gamma}$, the inequality follows 
since $\overline{\log \Gamma}$ surjects onto $\overline{\log \Gamma''}$.
For $d(\Gamma)$, if $H \le G$ and $H''$ is its image in $G''$,
then the subgroup 
$\Gamma/(\Gamma \intersect H)$ of $G/H$ surjects onto
the subgroup $\Gamma''/(\Gamma'' \intersect H'')$ of $G''/H''$,
and this implies the second inequality in
\begin{align*}
        d(\Gamma'') 
        &\le \dim H'' + \rk \Gamma'' - \rk(\Gamma'' \intersect H'') \\
        &\le \dim H'' + \rk \Gamma - \rk(\Gamma \intersect H) \\
        &\le \dim H + \rk \Gamma - \rk(\Gamma \intersect H).
\end{align*}
This holds for all $H \le G$, so $d(\Gamma'') \le d(\Gamma)$.
\item
The map of topological spaces $G(\Q_p) \to G''(\Q_p)$ is proper,
so the inverse image of $G''(\Q_p)_f$ is contained in $G(\Q_p)_f$;
this gives the first statement.
To prove $\del(\Gamma)=\del(\Gamma'')$,
first use~(1) to assume that $G \to G''$ is surjective,
so $G''=G/H$ for some finite $H \le G$.
By~(ii), $\del(\Gamma'') \le \del(\Gamma)$.
By~(3), $\del(\Gamma) \le \dim H + \del(\Gamma'') = \del(\Gamma'')$.
Thus $\del(\Gamma)=\del(\Gamma'')$.
\item
We may reduce to the case in which $\Gamma_1$ is a finite-index subgroup
of $\Gamma_2$.
Let $n=(\Gamma_2:\Gamma_1)$, so $n \Gamma_1 \le \Gamma_2 \le \Gamma_1$.
If $\Gamma_1 \le G(\Q)_f$, then $\Gamma_2 \le G(\Q)_f$.
Conversely, if $\Gamma_2 \le G(\Q)_f$, then $n\Gamma_1 \le G(\Q)_f$,
so $\Gamma_1 \le G(\Q)_f$ by~(iii) applied to $G \stackrel{n}\to G$.
In this case, (iii) gives $\del(n \Gamma_1) = \del(\Gamma_1)$,
and (ii) implies that both equal $\del(\Gamma_2)$.
\item
Let $\Gamma:=\Gamma_1 \times \Gamma_2$.
Since a product of compact open subgroups is a compact open subgroup,
we have $G_1(\Q_p)_f \times G_2(\Q_p)_f \le (G_1 \times G_2)(\Q_p)_f$.
(In fact, equality holds.)
Thus $\Gamma \le (G_1 \times G_2)(\Q)_f$.
The equality $\dim \overline{\Gamma} = 
\dim \overline{\Gamma_1} \times \dim \overline{\Gamma_2}$
follows from the definitions.
To prove the corresponding equality for $d$, we must show that
the infimum in the definition of $d(\Gamma)$
is realized for an $H$ of the form $H_1 \times H_2$
with $H_i \le G_i$.
Suppose instead that $K \le G_1 \times G_2$ realizes the infimum.
Let $\pi_1\colon G_1 \times G_2 \to G_1$ be the first projection.
Let $H_1=\pi_1(K)$.
Let $H_2=\ker(\pi_1|_K)$; view $H_2$ as a subgroup scheme of $G_2$.
Let $H=H_1 \times H_2$.
Thus $\dim K = \dim H$.
The exact sequence 
\[
        0 \to \Gamma_2 \intersect H_2 
        \to \Gamma \intersect K 
        \stackrel{\pi_1}\to \Gamma_1 \intersect H_1
\]
shows that $\rk(\Gamma \intersect K) \le \rk(\Gamma \intersect H)$,
so
\[
        \dim H + \rk \Gamma - \rk(\Gamma \intersect H)
        \le  \dim K + \rk \Gamma - \rk(\Gamma \intersect K).
\]
Thus $H$ too realizes the infimum in the definition of $d(\Gamma)$,
as desired.
\item
Apply~(ii) to the addition homomorphism $G \times G \to G$
and $\Gamma := \Gamma_1 \times \Gamma_2$, and use~(v).
\item
We have $\del(\Gamma) \le 1$ by~(2).
If $\dim \overline{\Gamma}=0$, then the finitely generated
torsion-free $\Z_p$-module $\overline{\log \Gamma}$ is of rank $0$, 
so it is $0$; therefore $\Gamma \subseteq \ker \log$, so $\Gamma$ is torsion,
contradicting the hypothesis $\rk \Gamma=1$.
If $d(\Gamma)=0$, then there exists $H \le G$ with $\dim H=0$ and $\rk \Gamma = \rk(\Gamma \intersect H)$;
then $H$ is finite, so $\rk(\Gamma \intersect H)=0$ and $\rk(\Gamma)=0$,
contradicting the hypothesis.
\item
By applying an element of $\GL_n(\Q) = \Aut \G_a^n$,
we may assume that $\Gamma = \Z^r \times \{0\}^{n-r} \le \Q^n = \G_a^n(\Q)$,
where $r:=\rk \Gamma$.
Using~(v), we reduce to the case $n=1$.
If $r=0$, then the result is trivial.
If $r=1$, use~(vii).
\end{romanenum}
\end{proof}

\section{Tori}
\label{S:tori}

\begin{lemma}
\label{L:subquotient}
Let $K$ be a Galois extension of $\Q$.
Let $\calG:=\Gal(K/\Q)$.
Then the representation $\OO_K^\times \tensor \C$ of $\calG$
is a subquotient of the regular representation.
\end{lemma}

\begin{proof}
Define the $\calG$-set 
$E:=\Hom_{\text{$\Q$-algebras}}(K,\C)$
of embeddings 
and the $\calG$-set $P$ of archimedean places of $K$,
the difference being that conjugate complex embeddings are identified in $P$.
Then $E$ is a principal homogeneous space of $\calG$,
and there is a natural surjection $E \to P$.
Therefore $\C^E$ is the regular representation
and the permutation representation $\C^P$ is a quotient of $\C^E$.
The proof of the Dirichlet unit theorem gives a
$\calG$-equivariant exact sequence
\[
        0 \to \OO_K^\times \tensor \R \stackrel{\log}\to \R^P \to \R \to 0
\]
so $\OO_K^\times \tensor \C$ is a subrepresentation of $\C^P$.
\end{proof}

\begin{proposition}
\label{P:torus rank}
Let $\calT$ be a group scheme of finite type over $\Z$
whose generic fiber $T:=\calT \times \Q$ is a torus.
Then
\begin{enumerate}
\item[(a)]
$\calT(\Z)$ is a finitely generated abelian group.
\item[(b)]
$\rk \calT(\Z) \le \dim T$.
\item[(c)]
If $\Gamma \le \calT(\Z)$, then $d(T,\Gamma) = \rk \Gamma$.
\end{enumerate}
\end{proposition}

\begin{proof}\hfill
\begin{enumerate}
\item[(a)]
For some number field $K$ and set $S$ of places of $K$,
we have $\calT \times \OO_{K,S} \isom (\G_m)_{\OO_{K,S}}^n$,
so $\calT(\OO_{K,S})$ is finitely generated by the Dirichlet $S$-unit theorem.
Therefore the subgroup $\calT(\Z)$ is finitely generated.
\item[(b)]
We may assume that $K$ is Galois over $\Q$.
Let $\calG = \Gal(K/\Q)$.
Let $X$ be the character group $\Hom(T_K,(\G_m)_K)$ of $T$.
Let $\chi_X$ be the character of the representation $X \tensor \C$ of $\calG$.
Let $\chi_K$ be the character of the representation $\OO_K^\times \tensor \C$
of $\calG$.
By Theorem~6.7 and Corollary~6.9 of \cite{Eisentraeger-thesis},
$\rk \calT(\Z) = (\chi_X,\chi_K)$.
On the other hand, $\dim T = \rk X = (\chi_X,\chi_{\reg})$,
where $\chi_{\reg}$ is the character of the regular representation of $\calG$.
The result now follows from Lemma~\ref{L:subquotient}.
\item[(c)]
First, $\calT(\Z)$ is contained in the compact open subgroup $\calT(\Z_p)$
of $T(\Q_p)$, so $d(T,\Gamma)$ is defined.
By~(2), $d(T,\Gamma) \le \rk \Gamma$.
To prove the opposite inequality,
we must show that for every subgroup scheme $H \le T$,
we have $\rk(\Gamma \intersect H) \le \dim H$.
By replacing $H$ by its connected component of the identity,
we may assume that $H$ is a subtorus of $T$.
Let $\calH$ be the Zariski closure of $H$ in $\calT$.
Then $\Gamma \intersect H \le \calH(\Z)$,
so $\rk(\Gamma \intersect H) \le \rk \calH(\Z) \le \dim H$
by~(b).
\end{enumerate}
\end{proof}

\begin{corollary}
\label{C:Leopoldt}
Let $K$ be a number field.
Let $T$ be the restriction of scalars $\Res_{K/\Q} \G_m$.
Let $\Gamma \le T(\Q) \isom K^\times$ correspond to $\OO_K^\times$.
Then Leopoldt's conjecture is equivalent
to a positive answer to Question~\ref{Q:over Q} for $\Gamma$.
\end{corollary}

\begin{proof}
Leopoldt's conjecture is the statement $\dim \overline{\Gamma} = \rk \Gamma$.
Let $\calT = \Res_{\OO_K/\Z} \G_m$.
By Proposition~\ref{P:torus rank}(c) applied to $\calT$, 
$d(T,\Gamma)=\rk \Gamma$.
So Leopoldt's conjecture is equivalent to 
$\dim \overline{\Gamma} = d(T,\Gamma)$.
\end{proof}

\begin{remark}
In effect, we have shown that Leopoldt's conjecture cannot be disproved simply
by finding a subtorus $H$ of $\Res_{K/\Q} \G_m$ containing a subgroup
of integer points of rank greater than $\dim H$.
This seems to have been known to experts, 
but we could not find a published proof.
\end{remark}

\section{Further questions}
\label{S:further questions}

\begin{question}
\label{Q:computable}
Is $d(\Gamma)$ computable in terms of $G$ and generators for $\Gamma$?
\end{question}

\begin{question}
If the answer to Question~\ref{Q:computable} is positive,
can $\dim \overline{\Gamma} = d(\Gamma)$ be verified in
each instance where it is true?
\end{question}

\begin{question}
\label{Q:number field}
Can one define a plausible generalization of $d(\Gamma)$
for the analogous situation
where $\Q$ and $\Q_p$ are replaced a number field $k$ and some
nonarchimedean completion $k_v$?
\end{question}

\begin{remark}
Applying restriction of scalars from $k$ to $\Q$ and then applying $d$
does not answer Question~\ref{Q:number field}:
it would instead predict the dimension of the closure
of $\Gamma$ in the product $\prod_{v|p} G(k_v)$
instead of in a single $G(k_v)$.
\end{remark}

\begin{remark}
If $G$ be a commutative group scheme of finite type over $\Q$,
we can consider also $G(\R)$, and define $G(\R)_f$ and $G(\Q)_f$.
The closure $\overline{\Gamma}$ 
of any subgroup $\Gamma \le G(\Q)_f$ in $G(\R)$ is a real Lie group.
The natural guess for $\dim \overline{\Gamma}$
seems now to be that it equals the dimension of the Zariski closure $H$
of $\Gamma$ in $G$; in other words, $\overline{\Gamma}$ 
should be open in $H(\R)$.
See \cite{Mazur1992}*{\S7} for a discussion of the abelian variety case.
\end{remark}

\section*{Acknowledgements} 

I thank Robert Coleman for suggesting 
the reference \cite{BourbakiLie1-3}*{III.\S7.6}.


\begin{bibdiv}
\begin{biblist}


\bib{BourbakiLie1-3}{book}{
  author={Bourbaki, Nicolas},
  title={Lie groups and Lie algebras. Chapters 1--3},
  series={Elements of Mathematics (Berlin)},
  note={Translated from the French; Reprint of the 1989 English translation},
  publisher={Springer-Verlag},
  place={Berlin},
  date={1998},
  pages={xviii+450},
  isbn={3-540-64242-0},
  review={MR1728312 (2001g:17006)},
}

\bib{Chabauty1941}{article}{
  author={Chabauty, Claude},
  title={Sur les points rationnels des courbes alg\'ebriques de genre sup\'erieur \`a l'unit\'e},
  language={French},
  journal={C. R. Acad. Sci. Paris},
  volume={212},
  date={1941},
  pages={882\ndash 885},
  review={MR0004484 (3,14d)},
}

\bib{Coleman1985chabauty}{article}{
  author={Coleman, Robert F.},
  title={Effective Chabauty},
  journal={Duke Math. J.},
  volume={52},
  date={1985},
  number={3},
  pages={765\ndash 770},
  issn={0012-7094},
  review={MR808103 (87f:11043)},
}

\bib{Eisentraeger-thesis}{book}{
  author={Eisentr\"ager, Anne Kirsten},
  title={Hilbert's tenth problem and arithmetic geometry},
  date={2003-05},
  note={Ph.D.\ thesis, University of California at Berkeley},
  pages={iii+79},
}

\bib{Mazur1992}{article}{
  author={Mazur, Barry},
  title={The topology of rational points},
  journal={Experiment. Math.},
  volume={1},
  date={1992},
  number={1},
  pages={35\ndash 45},
  issn={1058-6458},
  review={MR1181085 (93j:14020)},
}

\end{biblist}
\end{bibdiv}
\end{document}